\documentclass[12pt, reqno]{amsart}
\usepackage {amsmath,amssymb,verbatim}
\usepackage[pagebackref,pdftex,hyperindex]{hyperref}

\usepackage[margin=3cm]{geometry}

\newcommand{\A}{\mathbb{A}}
\newcommand{\B}{\mathbb{B}}
\newcommand{\C}{\mathbb{C}}

\newcommand{\G}{\mathbb{G}}
 
\newcommand{\N}{\mathbb{N}}

\renewcommand{\P}{\mathbb{P}}
\newcommand{\Q}{\mathbb{Q}}
\newcommand{\R}{\mathbb{R}}
\renewcommand{\S}{\mathbb{S}}

\newcommand{\cB}{\mathcal{B}}

\newcommand{\cE}{\mathcal{E}}

\newcommand{\cH}{\mathcal{H}}

\newcommand{\cJ}{\mathcal{J}}

\newcommand{\cL}{\mathcal{L}}

\newcommand{\cO}{\mathcal{O}}

\newcommand{\cX}{\mathcal{X}}

\renewcommand{\d}{\delta}
\newcommand{\e}{\varepsilon}

\renewcommand{\phi}{\varphi}

\newcommand{\eg}{{\rm e.g.\ }}

\renewcommand{\leq}{\leqslant}
\renewcommand{\geq}{\geqslant}
\newcommand{\abs}[1]{\left\lvert#1\right\rvert}
\newcommand{\norm}[1]{\left\|#1\right\|}

\renewcommand{\DH}{\mathrm{DH}}

\newcommand{\NA}{\mathrm{NA}}

\DeclareMathOperator{\CAT}{CAT}

\DeclareMathOperator{\lct}{lct}

\DeclareMathOperator{\PSH}{PSH}
\DeclareMathOperator{\Ric}{Ric}

\DeclareMathOperator{\supp}{supp}

\numberwithin{equation}{section}       

\newtheorem{prop} {Proposition} [section]
\newtheorem{thm}[prop] {Theorem} 
\newtheorem{dfn}[prop] {Definition}
\newtheorem{lem}[prop] {Lemma}
\newtheorem{cor}[prop]{Corollary}

\newtheorem{rem}[prop]{Remark}
\theoremstyle{remark}
\newtheorem*{ackn}{\bf{Acknowledgment}}

\newtheorem*{thmA}{\bf{Theorem A}} 
\newtheorem*{thmB}{\bf{Theorem B}}

\newtheorem*{dfn*}{\bf{Definition}}

\title[]{Geometric flow, multiplier ideal sheaves and optimal destabilizer for a Fano manifold} 
\date{\today} 

\author{Tomoyuki Hisamoto}
\address{Graduate School of Mathematics\\
  Nagoya University\\
  Furocho\\
  Chikusa\\
  Nagoya\\ 
  Japan}
\email{hisamoto@math.nagoya-u.ac.jp}

\begin{document}

\maketitle

\setcounter{tocdepth}{1}

\begin{abstract}
In \cite{Don05}, it was asked whether the lower bound of the Calabi functional is achieved by a sequence the normalized Donaldson-Futaki invariants. We answer to the question for the Ricci curvature formalism, in place of the scalar curvature. 
The principle is that the stability indicator is optimized by the multiplier ideal sheaves of certain weak geodesic ray asymptotic to the geometric flow. 
We actually prove it in the two cases: the inverse Monge-Amp\`ere flow and the K\"ahler-Ricci flow. 
\end{abstract}

\tableofcontents 

\section{Introduction}

Let $X$ be a Fano manifold. 
We are motivated to study how $X$ is far from K\"ahler-Einstein.  
To examine the curvature of each K\"ahler metric $\omega$ in the first Chern class $c_1(X)$ 
we make use of the normalized Ricci potential function $\rho$ which is characterized by 
\begin{equation}\label{Ricci potential}
\Ric \omega -\omega = dd^c\rho, ~~~
\int_X (e^\rho -1) \omega^n=0. 
\end{equation}
The volume $V= \int_X \omega^n$ is independent of $\omega$.   
The metric is K\"ahler-Einstein iff $\rho=0$ and it is equivalent to say that the scalar curvature is constant. 
For a general polarized manifold $(X, L)$, the famous Calabi functional measures how $\omega$ is far from constant scalar curvature and \cite{Don05} gives the lower bound in terms of his generalization of the Futaki invariant. 
For a Fano polarization $(X, -K_X)$ Ricci potential may work in place of the scalar curvature. 
In fact in analogy with Donaldson's lower bound, we have the inequality 

\begin{equation}\label{lower bound of the Ricci-Calabi functional}
\inf_{\omega} \bigg[ \frac{1}{V} \int_X (e^\rho -1)^2 \omega^n \bigg]^{\frac{1}{2}} 
\geq  \sup_{(\cX, \cL)} \frac{-D^\NA(\cX, \cL)}{\norm{(\cX, \cL)}_2}.  
\end{equation} 
Here $(\cX, \cL)$ runs through arbitrary test configurations of $(X, -K_X)$, $\norm{(\cX, \cL)}_2$ is the $L^2$-norm, and $D^\NA(\cX, \cL)$ is the non-Archimedean D-energy introduced in \cite{Berm16}, \cite{BHJ15}. 
We review the terminologies and a proof of (\ref{lower bound of the Ricci-Calabi functional}), in the next section. 
From the result of \cite{LX14}, positivity of $D^\NA(\cX, \cL)$ for every non-trivial test configuration is equivalent to the K-stability condition introduced by \cite{Don02}. 
The prototype of such inequalities already appears in geometric invariant theory (GIT for short), where it confronts square of the moment map with Hilbert-Mumford weights. 
For this reason we may call (\ref{lower bound of the Ricci-Calabi functional}) {\em moment-weight inequality}. 
The precise moment map picture was explained in \cite{Don15}, where the Kemp-Ness functional in GIT is translated into the (Archimedean) D-energy (\ref{D-energy}). 

In the scalar curvature setting 
Donaldson asked whether the equality holds in the above. 
In our setting, \cite{Yao17} recently proved that the equality in (\ref{lower bound of the Ricci-Calabi functional}) actually holds for toric Fano manifolds. 
If there exists a test configuration accomplishes the identity, it should be the optimal destabilizer in analogy with the Harder-Narasimhan and the Jordan-H\"older filtration for the vector bundles. 
The pioneering work \cite{N90} of Nadel already predicted that the certain multiplier ideal sheaf should serve as the destabilizing subsheaf of the vector bundle. See also \cite{PSS06}, \cite{Rub09}. 

In this paper we show that the equality holds in (\ref{lower bound of the Ricci-Calabi functional}) for general Fano manifolds, in virtue of adopting the Ricci potential formulation. 
The proof also gives a suggestion to the construction of the optimal destabilizer. 
Our new ingredient is the gradient flow of the D-energy. 
Using the $dd^c$-lemma we fix the reference metric $\omega_0$ and represent any other metric by a function $\phi$ so that $\omega= \omega_0 +dd^c\phi$ holds. 
The function $\phi$ is determined up to a constant and we consider $\rho=\rho_\phi$ or other quantities as functions in $\phi$. 
In terms of $\phi$ we introduce the {\em inverse Monge-Amp\`ere flow} 
\begin{equation}\label{inverse Monge-Ampere flow}
\frac{\partial}{\partial t} \phi = 1-e^\rho, 
\end{equation}
which imitates the Calabi flow in the scalar curvature setting. 
Although the long-time existence of the Calabi flow is still open question, we have the solution for (\ref{inverse Monge-Ampere flow}). 
This is one of the main results in our previous work \cite{CHT17}. 
Building on the Mabuchi geometry of space of K\"ahler metrics, especially on the technique exploited by \cite{DH17}, one can construct a weak geodesic ray $\Phi$ asymptotic to the flow. 
Blowing up the multiplier ideal sheaves $\cJ(m\Phi)$ for each $m \in \N$, we obtain a sequence of test configurations, which canonically approximates the geodesic ray. 
The technology here was paved by \cite{BBJ15} where they gave a variational approach to the celebrated result \cite{CDS15}. 
We will show that the equality of (\ref{lower bound of the Ricci-Calabi functional}) is then naturally achieved by the flow and these test configurations. 

\begin{thmA}[moment-weight equality]
Greatest lower bound of the Ricci-Calabi functional is given by a sequence of $L^2$-normalized non-Archimedean Ding energies, that is,  
\begin{equation*}
\inf_{\omega} \bigg[ \frac{1}{V} \int_X (e^\rho -1)^2 \omega^n \bigg]^{\frac{1}{2}} 
= \sup_{(\cX, \cL)} \frac{-D^\NA(\cX, \cL)}{\norm{(\cX, \cL)}_2}. 
\end{equation*} 
In fact the infimum is achieved by the inverse Monge-Amp\`ere flow (\ref{inverse Monge-Ampere flow}). 
The supremum is achieved by the test configurations which are defined as the blow-up of the associated multiplier ideal sheaves $\cJ(m\Phi)$.  
\end{thmA}   

Conjecturally the right-hand side would be the {\em maximum} attained by the optimal $(\cX, \cL)$, provided we slightly stretches the meaning of test configurations. 
Actually for toric Fano manifolds \cite{Yao17} constructed the optimal destabilizer as a possibly irrational but piecewise-linear convex function on the moment polytope. 
It implies that a single ideal sheaf can not generally optimize the stability indicator.  
Our proof shows that the weak geodesic ray attains the maximum in a suitable sense (see Remark \ref{geodesic ray achieves max}). 
It might be challenging to clarify whether the ray, constructed transcendentally  in the above, interpreted into certain algebraic singularities. 

Replacing the Ricci-Calabi functional with the H-functional $H(\omega)$, \cite{DS17} established the parallel equality and the corresponding optimal test configuration. 
In this formalism non-Archimedean D-energy is replaced with H-invariant of the test configurations.  
The idea of the present paper as well applied to this setting. 
In the final section we serve another simple proof of \cite{DS17}, Theorem 1.2, without using the deep result of \cite{CW14}, \cite{CSW15}. 

\begin{thmB}
Greatest lower bound of H-functional is given by a sequence of H-invariants: 
\begin{equation*}
\inf_{\omega} H(\omega) = \sup_{(\cX, \cL)} H(\cX, \cL). 
\end{equation*}
The infimum is achieved by the K\"ahler-Ricci flow. 
The supremum is achieved by the test configurations which is defined as the normalized blow-up of the associated multiplier ideal sheaves.  
\end{thmB} 
In the H-functional setting the maximum is attained by the test configuration constructed by \cite{CSW15}. 
Strictly speaking it is not a genuine test configuration but endowed with an irrational $\C^*$-action. In the terminology of \cite{DS17} it is called {\em $\R$-degeneration}. 
Our argument does not construct the $\R$-degeneration in the limit, while it shows that the maximum is effectively approximated by the associated multiplier ideal sheaves. 

It is known that $H(\cX, \cL)>0$ for all $(\cX, \cL)$ iff $X$ is D-semistable so that the H-invariant is weaker than the non-Archimedean D-energy. 
For example in the toric case the optimal destabilizer for the H-functional gives a product family while the optimal destabilizer for the Ricci-Calabi functional has jut two components in the central fiber. 
It indicates that H-optimizer corresponds to the Harder-Narasimhan filtration for vector bundles and D-optimizer even takes on a role of the Jordan-H\"older filtration for semistable ones. 
See \cite{CHT17} for the detail.  
The construction and comparison of these two destabilizers could be interesting from the viewpoint of birational geometry and should be investigated in the future work. 

Just when the author was going to post the preprint on arXiv, he was informed the appearing work \cite{Xia19} of M. Xia. 
It solves the {\em metrized} version of Theorem A, and of Donaldson's original conjecture for arbitrary compact K\"ahler manifolds, admitting finite energy geodesic rays in the supremum (so that the non-Archimedean D-energy is replaced with the {\em radial D-energy} of the geodesic ray).  
Note that in the scalar curvature setting the infimum requires singular $\phi$ because we do not have a smooth solution of the Calabi flow yet.  
Not a few ideas are in common and we even need \cite{Xia19}, Lemma 5.1 critically in proving Theorem A. 
We focus on the Fano case but instead answer to the original version of the question and moreover clarify the relation with multiplier ideals. 

\begin{ackn}
The author express his gratitude to Mingchen Xia. 
He especially pointed out the lack of the discussion in our previous version. 
After the communication I realized that the proof requires \cite{Xia19}, Lemma 5.1. 
I am grateful to A. Futaki who introduced \cite{Don05} at an early age of my Ph.D period and to Y. Odaka who enlightened me on the problem of optimal destabilizer. 
I would also like to thank T. Collins, E. Inoue and R. Takahashi for helpful discussions. 
This research was supported by JSPS KAKENHI Grant Number 15H06262 and 17K14185. 
\end{ackn}


\section{Preliminary} 

\subsection{Ricci curvature formulation}
We first give the variational setting for the K\"ahler-Einstein problem. 
Throughout the paper $X$ is an $n$-dimensional Fano manifold and $\omega$ denotes a K\"ahler metric whose cohomology class is the first Chern class $c_1(X)$. 
As in the introduction, we fix a reference metric to represent the metric by a function $\phi$. 
In words of the anti-canonical line bundle $-K_X$, 
one has a fiber metric $h_0$ with the Chern curvature $\omega_0$. 
Then $h_0e^{-\phi}$ defines another smooth fiber metric so that $\omega_\phi= \omega_0 +dd^c\phi$ gives the curvature. 
We freely chose appropriate description of the metric going back and forth between $\omega$, $\phi$ and the fiber metric $h_0e^{-\phi}$. 
Let us denote by $\cH=\cH(X, \omega_0)$, the collection of all smooth $\phi$ for which $\omega = \omega_\phi$ is strictly positive. 
We first introduce the Ricci-Calabi functional in $\phi \in \cH$, which is the curvature integration 
\begin{equation}
R(\phi):= \frac{1}{V}\int_X (e^\rho -1)^2 \omega^n. 
\end{equation}
This gives the analogue of the classical functional 
\begin{equation*}
C(\phi):= \frac{1}{V}\int_X (S_\omega -\hat{S})^2 \omega^n 
\end{equation*}
introduced by E. Calabi. In the above $\hat{S}$ denotes the mean value of the scalar curvature $S_\omega$. 
Compared to the Calabi functional, it is relatively recent result \cite{Don15} where the infinite-dimensional moment map picture for the Ricci-Calabi functional was given. 
The picture regards this functional as the square of the moment map and provides a natural prospect for the variational approach to the K\"ahler-Einstein problem. 
The sophisticated idea is to consider for each direction $\d \phi$ the pairing 
\begin{equation*}
\d \phi \mapsto \frac{1}{V}\int_X \d \phi (e^\rho -1) \omega^n. 
\end{equation*} 
It moreover defines an {\em exact} 1-form on $\cH$. 
Here we introduce the potential called D-energy: 
\begin{equation}\label{D-energy}
D(\phi)=L(\phi)-E(\phi):= -\log\frac{1}{V}\int_X e^{-\phi +\rho_0} \omega_0^n 
-E(\phi).  
\end{equation}
The above second term 
\begin{equation}
E(\phi) := \frac{1}{(n+1)V}\sum_{i=0}^{n} \int_X \phi \omega^{i} \wedge \omega_0^{n-i} 
\end{equation}
is called the Monge-Amp\`ere energy, because the differential is designed to be the Monge-Amp\`ere measure: $(dE)_\phi= V^{-1}\omega^n = V^{-1}\omega_\phi^n$. 
For the first term $L(\phi)$ we define the canonical probability measure 
\begin{equation} 
\mu_\phi := \frac{e^{-\phi+\rho_0}\omega_0^n}{\int_X e^{\rho_0}\omega_0^n} = V^{-1}e^\rho \omega^n  . 
\end{equation}
In terms of these probability measures the D-energy is characterized by the property $d_\phi D= \mu_\phi -\omega_\phi^n$. 
The critical point condition $d_\phi D=0$ is well-known equivalent to the K\"ahler-Einstein equation.  
The functional first appeared in \cite{BM85} and was written down to the above form by \cite{Din88}. 
It precisely plays the role of Kemp-Ness functional in the finite-dimensional GIT. 
As we review in the next subsection, the D-energy is convex with respect to the natural metric structure. 
Asking when the energy functional is proper we are naturally lead to the definition of D-stability. 

More recently in \cite{CHT17}, we studied the gradient flow of the D-energy 
\begin{equation*}
\frac{\partial}{\partial t} \phi 
= 1- e^\rho 
\end{equation*}
and particularly proved that the long-time solution exists. 

\begin{thm}[\cite{CHT17}, Theorem]
Given an initial data, the inverse Monge-Amp\`ere flow (\ref{inverse Monge-Ampere flow})  has the unique solution $\phi= \phi_t$ for all $t \in [0, \infty)$. 
Moreover, $E(\phi_t)$ is constant, $D(\phi_t)$ and $R(t)=\frac{d}{dt} D(\phi_t)$ are non-increasing. 
\end{thm} 
This is our key tool for speculating in which direction the D-energy {\em worstly} decays. 
In our notation the normalized K\"ahler-Ricci flow is written as 
\begin{equation*}
\frac{\partial}{\partial t} \phi 
= -\rho.   
\end{equation*}
The both flows converge to the K\"ahler-Einstein metric if it exists. On the other hand, the two flows show different behaviors when $\rho$ tends to be big. It is precisely the situation we are interested in.  

\subsection{Geodesic of finite energy metrics}

One remarkable property is that the Monge-Amp\`ere energy is affine and the D-energy  is convex along any geodesic for Mabuchi's $L^2$ structure. 
It strongly motivate us to exploit the general framework of convex optimization. 
In fact we may consider general $L^p$ structure for the space of K\"ahler metrics and especially need to consider $L^1$ geometry. 
First from $dd^c$-lemma any smooth function $u$ can be seen as a tangent vector at $\phi$. 
The $L^p$-norm 
\begin{equation}\label{Lp norm}
\norm{u}_p:= \bigg[ \frac{1}{V}\int_X \abs{u}^p \omega^n \bigg]^\frac{1}{p} 
\end{equation}
hence defines the distance $d_p$ on $\cH(X, \omega_0)$.  
The metric space is not complete so that even if the energy is proper existence of a minimizer is not guaranteed.  
Therefore the completion $\cE^p=\cE^p(X, \omega_0)$ comes to the forefront in the variational approach to the K\"ahler-Einstein problem. 
This is the main reason that we need to handle with a singular fiber metric $h_0e^{-\phi}$ for which $\phi$ is only assumed to be locally integrable. 
Such an $L^1$-function is called $\omega_0$-plurisubharmonic function (psh for short) if the curvature current $\omega=\omega_0 +dd^c\phi$ is semipositive. 
One can see \cite{BBGZ13}, \cite{BBEGZ11}, \cite{BBJ15}, \cite{Dar15}, \cite{Dar17a}, \cite{Dar17b}, and the textbook \cite{GZ17} for the developments in this area. 

Let us especially present the construction of $\cE^1$ which is indeed closely related with the Monge-Amp\`ere energy. 
It is well-known that we have the satisfactory definition of the product current $\omega_\phi^n$ and hence $E(\phi)$ for any bounded $\omega_0$-psh function $\phi$, by the celebrated work of Bedford-Taylor. 
To go further, for any $\omega_0$-psh $\phi$ we define the Monge-Amp\`ere energy as 
\begin{equation}
E(\phi) := \inf\bigg\{ E(\psi): \psi \in L^\infty \cap \PSH(X, \omega_0), \psi \geq \phi\bigg\} \in \R \cup \{-\infty\}. 
\end{equation} 
The function is called {\em finite energy} if $E(\phi)>-\infty$. 
We define the distance $d_1(\phi, \psi)$ of finite energy metrics approximating by  decreasing sequences of smooth $\omega_0$-psh functions. 

\begin{thm}[Special case of \cite{Dar15}, Theorem 2]
The space $(\cE^1(X, \omega_0), d_1)$ of all finite energy psh functions gives the completion of $(\cH(X, \omega_0), d_1)$. 
Moreover $d_1$ gives the coarsest refinement of the $L^1$-topology for psh functions so that the Monge-Amp\`ere energy is continuous. 
It follows that the D-energy is also continuous in this strong topology. 
\end{thm} 

Note that \cite{Dar15} gave a similar construction for general $(\cE^p(X, \omega_0), d_p)$. 

We next review a certain construction of geodesics. 
Henceforth we distinguish the geodesic $\phi^t$ from the inverse Monge-Amp\`ere flow $\phi_t$, by using the superscript. 
The singularity of the metric again inevitably appears if one considers a geodesic. 
Indeed the $L^2$-geodesic segment $\phi^t$ $(t \in [0, 1])$ in $\cE^2(X, \omega_0)$ has at best $C^{1, 1}$-regularity even if the endpoints are assumed to be smooth. 
For $L^1$-geodesic it is not even unique, as it was observed in \cite{Dar17a}. 
Given smooth endpoints there however exists a path $\phi^t$ which is characterized as the solution of the degenerate Monge-Amp\`ere equation and defines a geodesic for all $d_p$. 
We follow \cite{Bern11} for the construction. 
Let $\phi, \psi \in \cH$ and $a, b \in \R$. 
Take the complex variable $\tau$ of the annulus $A:= \{ \tau \in \C: e^{-b} < \abs{\tau} <e^{-a}\}$, as the translation of the time parameter $t= -\log \abs{\tau}$. 
Let us consider a function $\Psi \in \PSH(X\times A, p_1^*\omega_0 )$ with the boundary condition $\Psi(x, e^{-a}) \leq \phi(x), \Psi(x, e^{-b}) \leq \psi(x)$ and define the Peron-Bremermann type upper-semicontinuous envelope as 
\begin{equation}\label{PB envelope}
\Phi(x, \tau) 
:= {\sup}^* \Psi(x, \tau). 
\end{equation} 
The construction is also equivalent to the terminology {\em psh geodesic} in \cite{BBJ15}. 
As a standard fact, we have $\Phi(x, e^{-a}) = \phi(x), \Phi(x, e^{-b}) = \psi(x)$. 
Since we assume $\phi, \psi$ bounded $\Phi$ is also bounded. 
A standard argument of the pluripotential theory deduces that the $(n+1)$-variable Monge-Amp\`ere measure $(p_1^*\omega_0+dd^c_{x, \tau} \Phi)^{n+1}$ vanishes over $X\times A$. 
By the computation of \cite{Sem92} this is equivalent to say that $E(\phi^t)$ is affine. 
More generally, $E(\phi^t)$ is convex if $p_1^*\omega_0+dd^c_{x, \tau} \Phi \geq 0$. 
It follows that $\phi^t$ is weak geodesic for the $L^2$-structure. 
Moreover, by \cite{Dar15}, Theorem 4.17, $\phi^t$ defines a geodesic in the $L^p$-Finsler metric space $(\cE^p, d_p)$ for an arbitrary $p \geq 1$.  

It is rather recently proved by \cite{CTW17} that $\Phi$ has optimal $C^{1, 1}$-regularity for the smooth boundary data.  
From \cite{Dar15}, Remark 2.5, this geodesic of envelope form has a constant speed in $d_p$. It means that 
\begin{equation}\label{constant speed}
d_p(\phi^t, \phi^s) = d_p(\phi, \psi)\abs{\frac{t-s}{b-a}} 
\end{equation} 
for all $t, s$. Not all geodesics in $(\cE^1, d_1)$ satisfies the property. See also the discussion in \cite{Dar17a}, \cite{BBJ15}.  

Convexity of the D-energy is deeply related with H\"olmander $L^2$-estimate for the $\bar{\partial}$-equation and was  established by the seminal work of B. Berndtsson. 

\begin{thm}[\cite{Bern11}, Theorem 1.1]
Let $\Phi$ be a (possibly non-smooth) function on $X \times A$ such that $p_1^*\omega_0+dd^c_{x, \tau} \Phi \geq 0$ holds in the sense of current. 
Then for the associated segment $\phi^t$, $D(\phi^t)$ is a convex function. 
\end{thm} 

\subsection{Non-Archimedean energies and norms of the test configuration}

The famous Hilbert-Mumford criterion in GIT tells that properness of the Kemp-Ness functional is examined in each direction for a one-parameter subgroup. 
Once a given polarized manifold $(X, L)$ was embedded to the projective space each one-parameter subgroup of the linear transformation induces a degeneration $(\cX, \cL)$ called {\em test configuration}. 
It is then natural to ask the asymptotic behavior of D-energy along the degeneration. 
In the scalar curvature setting for a general polarized manifold 
\cite{Don02} first gave the intrinsic definition of a test configuration and introduced the Donaldson-Futaki invariant in relation to asymptotic behavior of the K-energy functional.  
The relationship between these invariants and energies for general test configurations was completed by \cite{Berm16}, \cite{BHJ15}, and \cite{BHJ16}. 

In this paper we first assume that any test configuration $(\cX, \cL)$ is a $\G_m$-equivariant family of $\Q$-polarized schemes, which is defined over the affine line $\A^1$. 
More generally we take account of the case when $\cL$ is relatively semiample. 
From the assumption the family is trivial outside of the origin and the generic fiber is isomorphic to the anti-canonical polarization $(X, -K_X)$. 
In terms of the equivariant isomorphism 
$\cX\vert_{\A^1\setminus \{0\}} \simeq \cX \times (\A^1\setminus \{0\})$, it is convenient to represent a point of $\cX\vert_{\A^1\setminus \{0\}}$ as $(x, \tau)$, where $x \in X$ and $\tau$ is the affine coordinate centered on $0 \in \A^1$. 
Moreover, $\cX$ may be assumed to be a normal variety. See \eg \cite{BHJ15} for the  detail discussion for the singularities. 

From an analytic point of view, each of the above degeneration can be regarded as the ray in the space of K\"ahler potentials. 
\begin{dfn}
Let us endow with $\cL$ a semipositive curvature fiber metric, defined over the unit disk $\B \subset \A^1$. 
It gives a function $\Phi$ on the punctured space such that the isomorphism $\cX\vert_{\B \setminus \{0\}} \simeq X\times (\B \setminus \{0\})$ translates the curvature form into $p_1^*\omega_0+dd^c_{x, \tau} \Phi$. 
We then define the associated ray $\phi^t$ on $\cH$ as 
\begin{equation*}
\phi^t(x):= \Phi(x, e^{-t}). 
\end{equation*} 
A ray of this form is called compatible with the test configuration. 
\end{dfn}
Any two compatible rays $\phi^t$ and $\psi^t$ can be considered to share the same asymptotic behavior because of the bound $\abs{\Phi -\Psi} \leq C $ which is uniform in $t$. 

For the same reason, one may even consider a non-smooth but bounded $\Phi$ for which $\omega^n=\omega_{\phi^t}^n$ and $E(\phi^t)$ is properly defined as we already mentioned. 
In \cite{Berm16}, inspired by \cite{Bern11}, a weak geodesic ray $\Phi$ associated with the test configuration was in fact constructed as the Peron-Bremermann envelope with the prescribed boundary value. 
Let us take a function $\Psi$ on $X \times (\B \setminus \{0 \})$ for which the corresponding fiber metric is extended to a singular fiber metric of $\cL$, so that the curvature is semipositive in the sense of current. 
The associated weak geodesic ray is defined as the upper-semicontinuous  envelope of $\Psi$ with the boundary condition $\Psi(x, 1) \leq \phi_{0}(x)$, which we denote 
\begin{equation}\label{associated weak geodesic ray}
\Phi(x, \tau) 
:= {\sup}^*  \Psi(x, \tau).  
\end{equation} 
Compare with the construction of the weak geodesic segment (\ref{PB envelope}). 
One can see that it is equivalent to the rays previously constructed in \cite{PS07}, \cite{CT08}, and \cite{RWN11}. 
Therefore the interest is the asymptotic behavior of $E(\phi^t)$ and $D(\phi^t)$ for the associated rays. 

Gluing $(\cX, \cL)$ with the trivial family we have the unique $\G_m$-equivariant family $(\bar{\cX}, \bar{\cL})$ defined over $\P^1$, so that the action as well is trivial in neighborhood of $\infty \in \P^1$. 
As it was compactified one can take the self-intersection number $\bar{\cL}^{n+1}$ which in fact gives the non-Archimedean counterpart of the Monge-Amp\`ere energy: 
\begin{equation}
E^\NA(\cX, \cL) := \frac{\bar{\cL}^{n+1}}{(n+1)V}. 
\end{equation}
Non-Archimedean D-energy is described as the log-canonical threshold 
\begin{align}
D^\NA(\cX, \cL) 
=L^\NA(\cX, \cL)-E^\NA(\cX, \cL) 
:= \lct_{(\bar{\cX}, \cB)}(\cX_0) -1 -\frac{\bar{\cL}^{n+1}}{(n+1)V}. 
\end{align}
Here the boundary divisor $\cB$ is uniquely determined by the property $\cB \sim_\Q -K_{\bar{\cX}/\P^1} -\bar{\cL}$ and $\supp \cB \subset \cX_0$. 
For the substantial non-Archimedean treatment, we refer \cite{BHJ15}, \cite{BFJ16}, \cite{BJ18}, and the survey article \cite{Bou18}.  
For our purpose it is sufficient to recall that it gives the slope of the Monge-Amp\`ere energy. 
For the D-energy the invariant $D^\NA(\cX, \cL)$ and the slope formula was built by R. Berman. See also the milestone works \cite{DT92a}, \cite{Tian97}. 
\begin{thm}[\cite{Berm16}, Theorem 3.11]
Let $(\cX, \cL)$ be a test configuration of a Fano manifold and take a bounded fiber metric of $\cL$, which is defined and has semipositive curvature over the unit disk.  
Then for the associated ray $\phi^t \in \cH$ one has 
\begin{equation*}
E^\NA(\cX, \cL) = \lim_{t \to \infty} \frac{E(\phi^t)}{t}, ~~~
D^\NA(\cX, \cL) = \lim_{t \to \infty} \frac{D(\phi^t)}{t}.  
\end{equation*}
\end{thm}

It shows that the non-Archimedean D-energy for the Ricci curvature formulation is just in parallel with the Donaldson-Futaki invariant defined by \cite{Don02} (equivalently, non-Archimedean K-energy defined by \cite{BHJ15}) for the scalar curvature formulation. 
In terms of the positivity of $D^\NA(\cX, \cL)$, one may define {\em D-stability} of $(X, -K_X)$. 
A Fano manifold $X$ is called D-semistable if $D^\NA(\cX, \cL) \geq 0$ for any test configuration. 
It is called D-polystable if moreover the equality holds precisely when $(\cX, \cL)$ is a product family (with a possibly non-trivial $\G_m$-action). 
As a result D-stability is equivalent to the K-stability. 
Indeed, from the work of \cite{LX14}, it is enough to consider so-called {\em special} test configuration in detecting the K-stability of a Fano manifold, and for these special test configurations $D^\NA(\cX, \cL)$ is equal to Donaldson-Futaki invariant. 
Therefore we observe: 
\begin{thm}[A consequence of \cite{CDS15} and \cite{LX14}]  
A Fano manifold admits a K\"ahler-Einstein metric iff it is D-polystable. 
\end{thm} 
We refer \cite{BBJ15} for the variational approach to this problem.  

In terms of the $\G_m$-action, $E^\NA(\cX, \cL)$ can be described as follows. 
Let us fix $k \in \N$ and write the weights $\lambda_1, \dots \lambda_{N_k}$ for the induced action on $H^0(\cX_0, k\cL_0)$. It is not so hard to see that 
\begin{equation}\label{weight description of E^NA}
E^\NA(\cX, \cL) = \lim_{k \to \infty} \frac{\sum_{j=0}^{N_k} \lambda_j}{kN_k}. 
\end{equation}
In particular we observe that replacing $\cL$ with line bundle $\cL +c\cX_0$ one has $E^\NA(\cX, \cL+c\cX_0)=E^\NA(\cX, \cL)+c$.  
Letting $\hat{\lambda}:= {N_k}^{-1}\sum_{j=0}^{N_k} \lambda_j$ we may further define the $L^p$-norm 
\begin{equation*}
\norm{(\cX, \cL)}_p:= \lim_{k \to \infty} \bigg[ \frac{\sum_{j=0}^{N_k} \abs{\lambda_j-\hat{\lambda}}^p}{k^{p}N_k} \bigg]^{\frac{1}{p}}, 
\end{equation*}
which is preserved by the above rescaling $\cL \mapsto \cL +c\cX_0$. 
The main result of \cite{His16} shows that these norms are equivalent to the $L^p$-norm of the associated weak geodesic ray. 
Notice that for the ray associated to the test configuration the best possible $C^{1, 1}$-regularity was established by \cite{CTW18}. See also \cite{PS10}. 
It follows that the time-derivative $\dot{\phi}^t$ is well-defined. 
\begin{thm}[\cite{His16}, Theorem 1.2] 
For the weak geodesic ray associated with a test configuration, we have 
\begin{equation*}
\norm{(\cX, \cL)}_p= \bigg[ \frac{1}{V}\int_X \abs{\dot{\phi}^t -E^\NA(\cX, \cL)}^p \omega_{\phi^t}^n  \bigg] ^\frac{1}{p}. 
\end{equation*}
In particular the right-hand side is independent of $t \in [0, \infty)$. 
\end{thm} 

Once the above results are accepted, the proof of the inequality (\ref{lower bound of the Ricci-Calabi functional}) is immediate. 
\begin{cor}[\cite{His16}, Theorem 1.3] 
For any K\"ahler metric $\omega$ in $c_1(X)$ and test configuration $(\cX, \cL)$, we have 
\begin{equation*}
\bigg[ \frac{1}{V}\int_X (e^\rho-1)^2 \omega^n \bigg]^{\frac{1}{2}} \geq \frac{-D^\NA(\cX, \cL)}{\norm{(\cX, \cL)}_2}. 
\end{equation*}
\end{cor}
\begin{proof} 
By constant rescaling we may assume $E^\NA(\cX, \cL)=0$ and the geodesic convexity implies 
\begin{align*}
-D^\NA(\cX, \cL) 
&= \lim_{t\to \infty} \frac{-D(\phi^t)}{t} \\ 
&\leq -\frac{d}{dt}\bigg\vert_{t=0}D(\phi^t)
= -\frac{1}{V}\int_X \dot{\phi}^0 (e^\rho -1) \omega_{\phi^0}^n.  
\end{align*}
Now the statement is a simple consequence of the Cauchy-Schwartz inequality. 
\end{proof}


\section{Construction of a weak geodesic ray and test configurations}

\subsection{Estimates for the inverse Monge-Amp\`ere flow}
In the sequel we assume that $X$ admits no K\"ahler-Einstein metric, otherwise the identity of Theorem A is trivial.  
We denote the solution of the inverse Monge-Amp\`ere flow (\ref{inverse Monge-Ampere flow}) by $\phi_t$ $(t\in [0, \infty))$ and  
fix any sequence $t_j \to \infty$. 
Although the solution of the flow is smooth, to consider geodesics we need the space $\cE^1$ of finite energy metrics.  

First we will take a geodesic segment $\phi_j^t \in \cE^1$ $(0\leq t \leq t_j)$, which joins $\phi_0$ to  
$\phi_{t_j}$. 
Our normalization of the Ricci potential yields that $E(\phi_t)$ is constant in $t$. For the supremum we have
\begin{lem}[Lemma 4.1 of \cite{CHT17}]\label{linear bound of the flow}
The flow is linearly bounded from above: 
\begin{equation*}
\phi_t \leq t +A.  
\end{equation*}
\end{lem} 
It follows that for any fixed $T$ Aubin's J-functional 
\begin{equation*}
J(\phi_t):= \frac{1}{V}\int_X \phi_t \omega_0^n -E(\phi_t)
\end{equation*} 
is bounded in $t \in[0, T]$. 
Notice that $J(\phi_t)$ is not bounded in $t \in [0, \infty)$, 
otherwise the flow converges to a weak minimizer of D-energy in $\cE^1$, namely the K\"ahler-Einstein metric. It achieves the equality in (\ref{lower bound of the Ricci-Calabi functional}). 
In other words, we have $\sup_X \phi_{t_j} \to +\infty$. 
In fact by \cite{Dar17b}, Corollary 4.14, $d_1$ is explicitly described in terms of $E$ as 
\begin{equation*}
d_1(\phi, \psi) = E(\phi) +E(\psi) -2E(P(\phi, \psi)),  
\end{equation*}
where  $P(\phi, \psi) \in \cE^1 $ is the upper envelope of $\omega_0$-functions $u$ such that $u \leq \phi, \psi$. 
The formula implies that 
\begin{equation*}
\sup_X \phi_{t_j}-C \leq d_1(\phi_0, \phi_{t_j}) \leq \sup_X \phi_{t_j}+C. 
\end{equation*}
This is comparable with \cite{DH17}, Theorem $1$ for the K\"ahler-Ricci trajectory. 

\subsection{Construction of a weak geodesic ray asymptotic to the flow}
Following the argument of \cite{DH17} we will show that, taking a subsequence if necessary, a particular choice (\ref{PB envelope}) of geodesics $\phi^t_j$ converges to a ray $\phi^t$ in $(\cE^1, d_1)$. 

The convergence argument focus on the relative entropy 
\begin{equation*}
H(\nu \vert \mu) := \int_X \log (\frac{\nu}{\mu})\nu 
\end{equation*}
of two given probabilistic measures $\nu$ and $\mu$, due to the fundamental compactness result for the subset with bounded entropy. 

\begin{thm}[\cite{BBEGZ11}] 
The subset 
\begin{equation*}
\bigg\{ \phi \in \cE^1: H(V^{-1}\omega_{\phi_t}^n\vert {V^{-1}\omega_0^n}) \leq C, ~~ \sup_X \phi = 0 \bigg\},  
\end{equation*}
 is compact in the $d_1$-topology. 
\end{thm} 

The role of the entropy in the K\"ahler-Einstein problem was observed by the thermodynamical formalism of \cite{Berm13}.
Firstly, D-energy is rather related with the entropy for the canonical probability measure $\nu= \mu_\phi$, by the following fact.  
\begin{prop}\label{Legendre transformation}
The relative entropy for the probabilistic measures $\nu$ and $\mu$ can be described as the Legendre dual: 
\begin{align*}
H(\nu \vert \mu) = \sup_{f \in C^0(X; \R)} \bigg[ \int_X f \nu -\log \int_X e^f \mu \bigg]. 
\end{align*}
\end{prop} 
The one-side inequality is obvious from Jensen's inequality and it is actually true for arbitrary lower-semicontinuous function $f$. 
This point will be more discussed in section $5$. 

The entropy of the Monge-Amp\`ere measure which we want to control forms the main term of Mabuchi's K-energy functional. 
Indeed using the integration by parts formula of \cite{Chen00b}, we may define the K-energy by 
\begin{equation}
M(\phi) 
= H(V^{-1}\omega_\phi^n\vert V^{-1}\omega_0^n) +\frac{1}{V}\int_X \phi \omega_\phi^n - E(\phi). 
\end{equation}
Note that the second term is controlled by $E$, since for general non-positive $\omega_0$-psh $\phi$ we have 
\begin{equation*}
(n+1)E(\phi) \leq \frac{1}{V}\int_X \phi \omega_\phi^n \leq E(\phi). 
\end{equation*}
It is immediate that the D-energy is monotone along the inverse Monge-Amp\`ere flow. 
In addition, Monge-Amp\`ere energy is conserved, from the normalization of $\rho$. 
Fortunately we have the monotonicity of the K-energy as well. 
\begin{lem}[\cite{CHT17}, Lemma 4.6]
Along the the inverse Monge-Amp\`ere flow, it holds 
\begin{equation*}
\frac{d}{dt} M(\phi) = -\frac{1}{V}\int_X \abs{\nabla \rho}^2 e^{\rho}\omega_\phi^n. 
\end{equation*} 
\end{lem} 
We deduce from the fact that for any fixed $T$ the entropy $H(V^{-1}\omega_{\phi_t}^n\vert V^{-1}\omega_0^n)$ is bounded in $t \in [0, T]$. 

From the compactness and (\ref{constant speed}) Ascoli's theorem implies that 
(after passing to a subsequence) $\phi^t_j$ converges to a ray $\phi^t$ in $(\cE^1, d_1)$. Moreover, for any fixed $T$, the convergent is uniform in $t \in [0, T]$. 
Since $t_j \to \infty $, $\phi^t$ is defined for $t \in [0, \infty) $. 
In fact by \cite{BBJ15}, Theorem 1.7, $\phi^t$ restricted to any interval $[a, b]$ is of envelope form (\ref{PB envelope}). 
Consequently the limit ray inherits the constant speed property: 
\begin{equation}\label{constant speed d1}
d_1(\phi^t, \phi^s) = d_1(\phi^0, \phi^1)\abs{t-s}. 
\end{equation} 
From the normalization we obtain 
\begin{equation}\label{sup=0}
\sup_X \phi^t \leq t+A 
\end{equation}
for any $t \in [0, \infty)$. 
On the other hand 
\begin{equation}\label{E=C}
\lim_{t\to \infty} \frac{E(\phi^t)}{t} = 0. 
\end{equation} 

For $p>1$, it is not clear from the construction whether $\phi^t$ is asymptotic to the flow, in the sense of \cite{DH17}.  
In general a ray $\phi^t$ is asymptotic to the curve $\phi_t$, if there exists $t_j \to \infty$ and constant speed geodesic segments $\phi_j^t$ $(t \in [0, t_j])$  connecting $\phi_0$ and $\phi_{t_j}$ such that for all $t$
\begin{equation*}
\lim_{j \to \infty} d_p(\phi_j^t, \phi^t) = 0. 
\end{equation*}
For the K\"ahler-Ricci flow \cite{DH17} derive the property from the Harnack estimate which is not estabilished for the inverse Monge-Amp\`ere flow. 
At present setting we will settle for the restrictive estimate: 

\begin{prop}[\cite{Xia19}, Lemma 5.1]\label{Xia}
For each $t$ we have $\phi^t \in \cE^2$ and 
\begin{equation*}
d_2(\phi^0, \phi^t) \leq \liminf_{j \to \infty} d_2(\phi^0, \phi_j^t). 
\end{equation*}
\end{prop} 
\begin{proof}
We sketch the proof.  
It fully exploits the $\CAT(0)$-property of $\cE^2$, which cannot be expected  for other complete length space $\cE^p$. 
In particular, we may generalize the notion of weak convergence in a Hilbert space to any complete $\CAT(0)$-space. See \cite{Bac14} for the general exposition. 

Let us fix any $t$.   
A standard argument of pluripotential theory shows that the non-increasing sequence 
\begin{equation*}
\psi_j^t := \sup_{k \geq j}{}^* \phi_k^t 
\end{equation*} 
converges almost everywhere to $\phi^t$. 
Since $\phi_j^t$ $(j=1, 2, \dots)$ are bounded in $\cE^2$, one can prove that so does $\psi_j^t$. 
The boundedness with monotonicity implies $\phi^t \in \cE^2$, by \cite{Dar15}, Lemma 4.16.  
Now \cite{BDL15}, Theorem 5.3 asserts that $\phi_j^t$ weakly converges to $\phi^t$. 
The point here is that the $d_1$-ball 
\begin{equation*}
B_\e (\phi) := \bigg\{ \psi \in \cE^2: d_1(\phi, \psi)<\e \bigg\} 
\end{equation*} 
is $d_2$-closed and $d_2$-convex. 
It follows that for any weakly convergent subsequence $\phi_{j_k}^t \to u^t$ $(k=1, 2, \dots)$
we have $u^t =\phi^t$. 
(From the $\cE^2$-boundedness we have at least one weakly convergent subsequence, by \cite{Bac14}, Proposition 3.1.2.) 
Indeed $\phi_{j_k}^t \in B_\e(\phi^t)$ for any sufficiently large $k$. 
Since $B_\e(\phi^t)$ is $d_2$-closed and $d_2$-convex, we conclude $u^t \in B_\e(\phi^t)$ by \cite{Bac14}, Lemma 3.2.1. 
The desired inequality follows from the fact that the distance function is lower-semicontinuous with respect to the weak convergence (\eg \cite{Bac14}, Corollary 3.2.4). 
\end{proof} 

In case $\phi^t \in \cE^p$, by \cite{DL18}, Theorem 1.2, we have $\phi^t$ as a  geodesic ray for any $(\cE^p, d_p)$. 
Moreover, each segment defines a unique geodesic ray when $p >1$. 
In particular it then has the constant speed for $d_p$. 
Such $\phi^t$ is distinguished as {\em finite energy geodesic} in \cite{DL18} and studied in view of geodesic stability. 

Summarizing up we obtain: 
\begin{thm}
Let $\phi_t$ be the inverse Monge Amp\`ere flow and $\phi_j^t$ $(t \in [0, t_j])$ be the weak geodesic ray of the envelope form (\ref{PB envelope}) so as to connect $\phi_0$ to $\phi_{t_j}$. 
Then there exists a ray $\phi^t$ of envelope form such that $\lim_{j \to \infty} d_1(\phi_j^t, \phi^t) = 0$ for each $t$. 
As a result $\phi^t$ is a geodesic for all $(\cE^p(X, \omega_0), d_p)$ and satisfies (\ref{constant speed d1}), (\ref{sup=0}) and (\ref{E=C}).  
\end{thm} 

\begin{rem} 
Provided the Harnack-type estimate for the inverse Monge-Amp\`ere flow was established we may apply \cite{DH17}, Theorem $3.2$ and obtain the ray directly. 
Such an estimate is highly non-trivial, as it implies the linear lower bound of $\phi_t$, or equivalently, the upper bound of the Ricci potential $\rho$. 
\end{rem} 

\subsection{Approximative test configurations}\label{Approximative test configurations}

Next we follow \cite{BBJ15} to construct a canonical sequence of test configurations which approaches to $\phi^t$. 
It can be seen as the non-Archimedean analogue of Demailly's approximation \cite{Dem92} for a psh function. 

Changing variables as 
\begin{equation*}
\Phi(x, e^{-t}) := \phi^t(x), 
\end{equation*}
we obtain the $\S^1$-invariant function $\Phi$ on $X\times (\B \setminus \{0\})$, which is actually $p_1^*\omega_0$-psh. 
From (\ref{sup=0}) $\hat{\Phi}:=\Phi+\log\abs{\tau}$ is uniquely extended to a $p_1^*\omega_0$-psh function on $X\times \B$. 
Since $\phi^t \in \cE^1$, the Lelong number is concentrated in $X\times \{0\}$. 
Moreover (\ref{sup=0}) implies that even the generic Lelong number along $X\times \{0\}$ is zero. 
Therefore, support of the $\S^1$-invariant multiplier ideal sheaf $\cJ(m\hat{\Phi})$ 
is properly contained in $X\times \{0\}$, so that we have the normalized blow-up $\rho_m\colon \cX_m \to X\times \C$. 
It would be remarkable that the argument really requires the definition of multiplier ideal sheaves for general plurisubharmonic functions, since $\Phi$ may have non-algebraic singularities. 
Let $E_m$ be the exceptional divisor. We fix some $m_0 \in \N$ and set the line bundle as 
\begin{equation}\label{line bundle of the test configuration}
\cL_m:=\rho_m^*p_1^*(-K_X) -\frac{1}{m+m_0}E_m+\frac{m}{m+m_0}\rho_m^*\cX_{m, 0}.  
\end{equation}
The number $m_0$ is chosen so that 
$\cO(-(m+m_0)p_1^*K_X)\otimes \cJ(m\hat{\Phi})$
is globally generated for all $m \geq 1$. See \cite{BBJ15}, Lemma 5.6. 
The term involving the central fiber $\cX_{m, 0}$ preserves the linearly equivalence of $\cL_m$ and only adjusts the $\G_m$-action. 
The constructed semiample test configuration $(\cX_m, \cL_m)$ satisfies the following continuity property, which is crucial for their variational approach to the K\"ahler-Einstein problem. 

\begin{thm}(\cite{BBJ15}, Theorem 5.4, Lemma 5.7 and 5.8)\label{continuity of D^NA}
For the above constructed weak geodesic ray and test configurations the upper-semicontinuity  
\begin{align*}
\limsup_{m\to \infty} D^\NA(\cX_m, \cL_m) \leq \lim_{t \to \infty} \frac{D(\phi^t)}{t} 
\end{align*} 
holds. 
Moreover, if $\phi^t$ is maximal in the sense of \cite{BBJ15}, Definition 6.5, we have the continuity 
\begin{align*}
\lim_{m\to \infty} D^\NA(\cX_m, \cL_m) = \lim_{t \to \infty} \frac{D(\phi^t)}{t}. 
\end{align*} 
\end{thm} 
The result tells that unlike the original version of Demailly's approximation, for a general weak geodesic ray the multiplier ideal sheaf construction can not reach $\phi^t$ in the limit. 
Since our setting looks slight different from \cite{BBJ15}, let us repeat this part of the proof. 
A similar idea will appear when we compare the $L^2$-norms in the last part of the proof of Theorem A. 
We take an $\S^1$-invariant, non-negatively curved smooth (or more generally bounded) fiber metric of the $\Q$-line bundle $\cL_m$ on $X \times \B$. 
It defines a $p_2^*\omega_0$-psh function $\Phi_m$ endowed with the analytic singularity of $\cJ(m\Phi)^{\frac{1}{m+m_0}}$. 
This is the reason why we adjusted the line bundle by $\cX_0$, in (\ref{line bundle of the test configuration}). 
Using Demailly's approximation theorem locally, we have the estimate 
\begin{equation*}
\Phi_m \geq \Phi -C_{m, r}
\end{equation*}
on the shrunken area $\B(0, r)\times X$. 
The positive constants $C$ and $r$ are independent of $m$. 
Since the Monge-Amp\`ere energy is non-decreasing, it follows  
\begin{align*}
E^\NA(\cX_m, \cL_m) 
&= \lim_{t\to \infty} \frac{E(\phi^t_m)}{t} \\
&\geq \lim_{t\to \infty} \frac{E(\phi^t-C_{m, r})}{t}
=\lim_{t \to \infty}\frac{E(\phi^t)}{t} =0.  
\end{align*}
The key point in the above is the Ohsawa-Takegoshi $L^2$-extension theorem \cite{OT87} used in Demailly's approximation. 
Notice that such a uniform lower bound estimate of the Bergman kernel already forms a basis of the celebrated work \cite{CDS15} (see also \cite{Tian15}).  

\begin{rem}
We may ask whether the constructed weak geodesic ray asymptotic to the inverse Monge-Amp\`ere flow is maximal in the sense of \cite{BBJ15}. 
For the proof of Theorem A, however, we do not require the maximality. 
\end{rem} 
 
 
\section{Proof of the moment-weight equality} 
 
\subsection{Test configurations almost destabilize $X$} 
The inverse Monge-Amp\`ere flow satisfies  
\begin{equation*}
\frac{d}{dt} D(\phi_t) = -\frac{1}{V} \int_X (e^\rho-1)^2 \omega_\phi^n =R(\phi_t) 
\end{equation*}
and $R(\phi_t)$ is non-increasing, as a property of the gradient flow.  
In particular $\frac{d}{dt} D(\phi_t) \leq 0$ and the convexity assures $\lim_{t \to \infty} \frac{D(\phi_t)}{t} \in [-\infty, 0]$ exists. 
It then follows 
\begin{align*}
\lim_{t \to \infty} \frac{D(\phi_t)}{t} 
=\lim_{j \to \infty}  \frac{D(\phi_{t_j})}{t_j} 
=\lim_{j \to \infty} \frac{D(\phi_j^{t_j})}{t_j}. 
\end{align*}
Since D-energy is convex along any geodesic, 
for any fixed $T$ we have 
\begin{equation*}
\lim_{j \to \infty}  \frac{D(\phi_j^{t_j})}{t_j}
\geq  \frac{D(\phi_j^{T})}{T}. 
\end{equation*} 
The convergence of $\phi_j^t$ to $\phi^t$ in $(\cE^1, d_1)$ then yields  
\begin{equation*}
\lim_{t \to \infty} \frac{D(\phi_t)}{t} 
\geq \frac{D(\phi^T)}{T}. 
\end{equation*} 
Letting $T \to \infty $, Theorem \ref{continuity of D^NA} now implies 
\begin{prop}\label{D^NA is negative}
Let $\phi_t$ be the inverse Monge-Amp\`ere flow and $\phi^t$ be a weak geodesic ray asymptotic to the flow. 
For the test configurations which canonically approximates $\phi^t$ we have 
\begin{equation*}
0\geq \lim_{t \to \infty} \frac{D(\phi_t)}{t} 
\geq \limsup_{m\to \infty } D^\NA(\cX_m, \cL_m). 
\end{equation*}
\end{prop} 
The proposition already shows that $(\cX_m, \cL_m)$ almost destabilize $X$. 
To get more precise upper bound of $D^\NA(\cX_m, \cL_m)$, we prepare computing the differential of the energy along the flow. 
\begin{lem}\label{Cauchy-Schwartz}
Along the inverse Monge-Amp\`ere flow 
we have 
\begin{align*}
-\frac{d}{dt} D(\phi_t) 
&= 
-\frac{1}{V}\int_X \dot{\phi_t} (e^{\rho_t}-1) \omega_{\phi_t}^n \\
&= \bigg[\frac{1}{V}\int_X (\dot{\phi_t})^2 \omega_{\phi_t}^n \bigg]^{\frac{1}{2}}
\bigg[ \frac{1}{V}\int_X (e^{\rho_t}-1)^2 \omega_{\phi_t}^n \bigg]^\frac{1}{2}. 
\end{align*}
\end{lem} 
Proof is immediate. Indeed, from the very definition of the inverse Monge-Amp\`ere flow, the equality holds in the Cauchy-Schwartz inequality. 

\begin{rem}
It is natural to expect the optimal destabilizer for general $L^p$-norm. 
See also \cite{DL18}, Theorem 1.6.  
In our argument, however, Lemma \ref{Cauchy-Schwartz} apparently requires $L^2$-norm. 
In addition, the proof of Proposition \ref{Xia} relies on the $\CAT(0)$-property of $d_2$. 
For a Fano manifold with no zero holomorphic vector fields, existence of the K\"ahler-Einstein metric is equivalent to the uniform stability with respect to the $L^1$-norm, as a result of \cite{BBJ15}.  
Note that existence of $L^p$-destabilizer does not contradicts to the fact. 
\end{rem} 

\subsection{Comparison of the norms}
In regard with Lemma \ref{Cauchy-Schwartz} 
we thus finally should study the $L^2$-norm 
\begin{equation*}
\norm{\dot{\phi}_t}_2:= \bigg[\frac{1}{V}\int_X (\dot{\phi_t})^2 \omega_{\phi_t}^n \bigg]^{\frac{1}{2}}. 
\end{equation*}

For the inverse Monge-Amp\`ere flow we have $\dot{\phi_t} =1-e^\rho$ so that $\norm{\dot{\phi}_t}_2=R(\phi_t)^{\frac{1}{2}}$ is non-increasing. 
Note that the weak geodesic ray $\phi^t \in \cE^2$ is possibly apart from any test configurations and it might be not even $C^1$. 
For this reason we make use of the choice of $\phi^t$ and regard the norm $\norm{\dot{\phi}^t}_2$ as follows. 
\begin{dfn}\label{definition of the norm}
For a weak geodesic ray $\phi^t \in \cE^2(X, \omega_0)$ with constant speed, we define the $L^2$-norm as 
\begin{equation*}
\norm{\dot{\phi}^t}_2:= 
\lim_{t \to \infty} \frac{d_2 (\phi^0, \phi^t)}{t}
= \frac{d_2 (\phi^0, \phi^t)}{t}. 
\end{equation*}
\end{dfn}  
It is of course consistent with the definition for a differentiable $\phi^t$. 
Observe that $\norm{\dot{\phi}^t}_2$ is constant in $t$ and moreover it is independent of the initial metric $\phi^0$.
Let us now take $\phi^t$ as in section 3.2. 

\begin{lem}
For the above norms we have 
\begin{equation*}
\norm{\dot{\phi}_t}_2 \geq \norm{\dot{\phi}^t}_2. 
\end{equation*}
\end{lem} 
\begin{proof} 
If $\norm{\dot{\phi}_t}_2 <\norm{\dot{\phi}^t}_2$ for some $t$, the above monotonicity implies that $\norm{\dot{\phi}_t}_2  <\norm{\dot{\phi}^t}_2$ holds for any sufficiently large $t \geq T$. 
Proposition \ref{Xia} implies that 
the right hand side is bounded from above as 
\begin{equation*}
\norm{\dot{\phi}^t}_2
= \frac{d_2(\phi^0, \phi^t)}{t}
\leq \liminf_{j\to \infty} \frac{d_2(\phi^0, \phi_j^t)}{t}
=\liminf_{j\to \infty} \norm{\dot{\phi}_j^t}_2. 
\end{equation*}
They are all independent of $t$. 
Therefore we may take $\e>0$ such that $\norm{\dot{\phi}_t}_2  +\e <\norm{\dot{\phi}_j^t}_2 $ for all $j$ and $t \geq T$. 
It implies $d_2(\phi_0, \phi_{t_j}) < d_2(\phi_0, \phi_j^{t_j})$ for a sufficiently large $j$. 
On the other hand, the $L^2$-geodesic connecting two metrics is unique by \cite{Dar17b}, Lemma 6.12, so that it has minimal length in all paths. 
It contradicts to our choice of $\phi_j^t$ which is $L^p$-geodesic for any $p\geq 1$.  
We conclude $\norm{\dot{\phi}_t}_2 \geq \norm{\dot{\phi}^t}_2$. 
\end{proof}

Now we take a $p_1^*\omega_0$-psh function $\Phi_m$ as the weak geodesic ray associated to $(\cX_m, \cL_m)$, and compare $\norm{\dot{\phi}^t}_2$ with $\norm{\dot{\phi}_m^t}_2$. 
Recall that the weak geodesic ray associated to the test configuration has $C^{1, 1}$-regularity by \cite{PS10}, \cite{CTW18}. 
It implies that the norm $\norm{\dot{\phi}_m^t}_2$ is well-defined. 
Let us invoke the following Lidskii type inequality. 
\begin{thm}[\cite{DLR18}, Theorem 5.1]\label{Lidskii type inequality}
For any $u, v, w \in \cE^p(X, \omega_0)$ with $u\geq v\geq w$ 
we have
\begin{equation*}
d_p(v, w)\leq d_p(u, w) -d_p(u, v). 
\end{equation*} 
\end{thm} 
\begin{lem}
For the associated weak geodesic rays $\phi_m^t$ we have 
\begin{equation*}
\norm{\dot{\phi}^t}_2 \geq \norm{\dot{\phi}_m^t}_2. 
\end{equation*} 
\end{lem} 
\begin{proof}
Since $\Phi_m$ comes from a bounded fiber metric of $\cL_m$, it encodes the analytic singularity $\cJ(m\Phi)^{\frac{1}{m+m_0}}$. 
Again by using Demailly's approximation theorem locally, we have $\Phi_m \geq \Phi-C_{m, r}$. 
Since $\phi_m^t$ is bounded from above and $\phi^0$ is smooth there exists a constant $B_m$ such that 
$\phi^0 +B_m \geq \phi_m^t$. 
We are ready to apply Lidskii type inequality: Theorem \ref{Lidskii type inequality} to these functions and get 
\begin{align}\label{Lemma 6.10 of Darvas}
 d_2 (\phi^0+B_m+C_{m, r}, \phi_m^t+C_{m, r})  
\leq d_2 (\phi^0+B_m+C_{m, r}, \phi^t).  
\end{align}
It follows from the triangle inequality that 
\begin{equation}
\norm{\dot{\phi}^t}_2
=\lim_{t \to \infty} \frac{d_2 (\phi^0, \phi^t)}{t} 
\geq \lim_{t \to \infty} \frac{d_2 (\phi^0, \phi_m^t)}{t} 
=\norm{\dot{\phi}_m^t}.  
\end{equation}
\end{proof} 
Combining the results all together, we obtain 
\begin{align*}
\liminf_{m\to \infty } -D^\NA(\cX_m, \cL_m) 
&\geq \lim_{t \to \infty} \frac{-D(\phi_t)}{t} \\
&= \lim_{t \to \infty} \norm{\dot{\phi}_t}_2
\bigg[ \frac{1}{V}\int_X (e^{\rho_t}-1)^2 \omega_{\phi_t}^n \bigg]^\frac{1}{2} \\
& \geq \norm{\dot{\phi}_m^t}_2 \lim_{t \to \infty}\bigg[ \frac{1}{V}\int_X (e^{\rho_t}-1)^2 \omega_{\phi_t}^n \bigg]^\frac{1}{2} 
\end{align*} 
for all $m$. 
For a while we denote $\e_m:= E^\NA(\cX_m, \cL_m)$ which is nonnegative, as a consequence of Theorem \ref{continuity of D^NA}. 
Recall that the norm $\norm{(\cX_m, \cL_m)}_2= \norm{\dot{\phi}_m^t-\e_m}_2$ slightly differs from the above $\norm{\dot{\phi}_m^t}_2$, however, we observe 
\begin{equation*}
\norm{\dot{\phi}_m^t-\e_m}_2^2= \norm{\dot{\phi}_m^t}_2^2 -\e_m^2 \leq \norm{\dot{\phi}_m^t}_2^2  
\end{equation*}
and hence conclude  
\begin{equation*}
\liminf_{m\to \infty } \frac{-D^\NA(\cX_m, \cL_m)}{\norm{(\cX_m, \cL_m)}_2}
\geq \lim_{t \to \infty}  \bigg[ \frac{1}{V}\int_X (e^{\rho_t}-1)^2 \omega_{\phi_t}^n \bigg]^\frac{1}{2}. 
\end{equation*}
The last inequality completes the proof of Theorem A. 
 
\begin{rem}\label{geodesic ray achieves max}
The above proof of Theorem A shows that 
\begin{equation*}
\inf_{\omega} \bigg[ \frac{1}{V}\int_X (e^\rho-1)^2 \omega^n \bigg]^{\frac{1}{2}} 
\leq \frac{1}{\norm{\dot{\phi}^t}_2}\lim_{t \to \infty} \frac{-D(\phi^t)}{t} 
\end{equation*} 
holds for a weak geodesic ray asymptotic to the inverse Monge-Amp\`ere flow. 
If the non-Archimedean potential $\Phi^\NA$ induced by $\phi^t$ is maximal in the sense of \cite{BBJ15}, the radial D-energy $\lim_{t \to \infty} t^{-1}D(\phi^t) $
equals to the non-Archimedean D-energy of $\Phi^\NA$. 
\end{rem} 
As a consequence of \cite{Li17}, 
the lower bound of the Calabi functional is zero iff $X$ is $D$-semistable (see also \cite{BBJ15}). 
We may restate the result in terms of the inverse Monge-Amp\`ere flow. 
\begin{cor}
Any Fano manifold $X$ admits a K\"ahler metric with arbitrary small Ricci potential, otherwise a weak geodesic ray $\phi^t$ asymptotic to the inverse Monge-Amp\`ere flow has negative slope: 
\begin{equation*}
\lim_{t\to \infty} \frac{D(\phi^t)}{t} <0. 
\end{equation*} 
In particular, there exists a test configuration $(\cX, \cL)$ with $D^\NA(\cX, \cL) <0$. 
\end{cor} 


\section{The K\"ahler-Ricci flow case}

\subsection{H-functional and H-invariant}
Recall that for $\phi \in \cH(X, \omega_0)$ we define the canonical probability measure 
\begin{equation}
\mu_\phi := \frac{e^{-\phi+\rho_0}\omega_0^n}{\int_X e^{\rho_0}\omega_0^n} . 
\end{equation}
In terms of the probability measure the D-energy is characterized by the property $d_\phi D= \mu_\phi -\omega_\phi^n$. 
The $H$-functional of \cite{He16} is described as the relative entropy functional: 
\begin{equation*}
H(\omega_\phi) := H(\mu_\phi \vert V^{-1}\omega_\phi^n ). 
\end{equation*}
See section $3$ and especially Proposition \ref{Legendre transformation} for our convention about the relative entropy. 
The functional first appeared in \cite{DT92b} and has played an important role in the study of K\"ahler-Ricci flow. 
As a consequence of Pinsker's inequality it is at least bounded from below by the $L^1$-version of the Ricci-Calabi functional: 
\begin{equation}
\sqrt{2H(\omega)} \geq \frac{1}{V}\int_X \abs{e^\rho-1} \omega^n.  
\end{equation}
Following \cite{DS17}, let us introduce the algebraic H-invariant of a test configuration as 
\begin{align*}
H(\cX, \cL)
&=-L^\NA(\cX, \cL) +F(\cX, \cL) \\
&:=-L^\NA(\cX, \cL) +\lim_{k \to \infty} \bigg[ -\log \frac{1}{N_k}\sum_{j=1}^{N_k} e^{-\frac{\lambda_j}{k}}\bigg], 
\end{align*}
where $\lambda_1, \dots, \lambda_{N_k}$ is the weights of the induced $\C^*$-action on $H^0(\cX_0, k\cL_0)$. 
Comparing with weight description of the non-Archimedean Monge-Amp\`ere energy (\ref{weight description of E^NA}), we observe 
\begin{equation}
H(\cX, \cL)\geq -D^\NA(\cX, \cL). 
\end{equation} 
Indeed H-invariant is weaker than the non-Archimedean D-energy. 
The Fano manifold $X$ satisfies $H(\cX, \cL) > 0$ for any non-trivial test configurations iff it is D-semistable (see \cite{DS17}, Lemma 2.5). 
Unfortunately, the second term $F$ is  in nature more transcendental than $E^\NA$. 
It does not simply correspond to the classical energy of metrics. 
At least for the associated weak geodesic ray, one may observe that 
the ``virtual slope"
\begin{equation*}
F(\dot{\phi}^t) := -\log \frac{1}{V}\int_X e^{-\dot{\phi}^t} \omega^n 
\end{equation*} 
gives $F(\cX, \cL)$, due to the following result. 
\begin{thm}[\cite{His16}]
Let $(\cX, \cL)$ be a test configuration and $\phi^t$ the associated $C^{1, 1}$-weak geodesic ray. Then the pushed-forward probability measure 
\begin{equation*}
\DH(\cX, \cL):=\dot{\phi}^t_*(V^{-1}\omega_{\phi^t}^n) 
\end{equation*} 
is independent of the initial data $\phi^0$ and $t$. 
Moreover, we have the convergence of the spectral measure: 
\begin{equation*}
\frac{1}{N_k}\sum_{j=1}^{N_k} \delta_\frac{\lambda_j}{k} \to \DH(\cX, \cL). 
\end{equation*} 
\end{thm} 
It then follows from \cite{Berm16}, Theorem 3.11 the slope formula 
\begin{equation}\label{slope of H}
H(\cX, \cL) = \lim_{t\to \infty} \bigg[ -\frac{L(\phi^t)}{t} + F(\dot{\phi}^t)\bigg]. 
\end{equation}

Lower bound of the H-functional is achieved by the supremum of these (unnormalized) H-invariant.  
\begin{thm}[\cite{DS17}, Theorem 1.2]\label{lower bound of the H-functional}
For a Fano manifold we have 
\begin{equation*}
\inf_{\omega} H(\omega) = \sup_{(\cX, \cL)} H(\cX, \cL). 
\end{equation*}
\end{thm} 

The one-side inequality is easier to see from (\ref{slope of H}). 
Indeed if we take $f:=-\dot{\phi}$ in Proposition \ref{Legendre transformation} the associated weak geodesic ray satisfies 
\begin{align*}
 H(\omega_{\phi^0}) 
&\geq  -\int_X \dot{\phi}^0 \mu_{\phi^0} -\log \frac{1}{V} \int_X e^{-\dot{\phi}^0} \omega_{\phi^0}^n \\
&\geq -\frac{d}{dt}L(\phi^t) +F(\dot{\phi}^t) 
\end{align*}
for any choice of the initial metric $\phi^0$. 

The quantity $\inf_\omega H(\omega)$  is equivalent to the supremum of Perelman's $\mu$-entropy. For a smooth function $f$ satisfying  
\begin{equation*}
\int_X e^{-f} \omega^n = V,  
\end{equation*}
we define the W-functional as 
\begin{equation*}
W(\omega, f) := \int_X (S_\omega+\abs{\nabla f}^2 +f )e^{-f} \omega^n. 
\end{equation*}
Perelman's $\mu$-entropy is then defined to be the infimum: 
\begin{equation*}
\mu(\omega) := \inf_{f}W(\omega, f) \leq nV.  
\end{equation*}
\begin{thm}[\cite{DS17}, Theorem 4.2]\label{mu-entropy} 
\begin{equation*}
\sup_\omega \mu(\omega) = nV -\inf_{\omega} H(\omega). 
\end{equation*}
\end{thm} 
We also remark the relation with the greatest lower bound of the Ricci curvature 
\begin{equation*}
R(X) := \sup \bigg\{ r \in [0, 1]: \Ric \omega \geq r \omega \bigg \}.  
\end{equation*}
It was shown in \cite{BBJ15} and \cite{CRZ18} that $R(X)= \min\{\d_X, 1\}$, where $\d_X$ is the $\delta$-invariant of Fujita-Odaka inspired by \cite{Berm18}.    
See \cite{FO16} for the definition and \cite{BJ18} for the non-Archimedean interpretation. 
In particular, it follows that $R(X) = 1$ iff $X$ is D-semistable. 

\begin{prop}
If $R(X) > 1/4\pi$, we have $nVR(X) \leq \sup_\omega \mu(\omega) \leq nV$ 
and so that 
\begin{equation*}
\inf_{\omega} H(\omega) \leq  nV(1-R(X)). 
\end{equation*} 
\end{prop}  
\begin{proof}
Let us take any $r < R(X)$ close to $R(X)$ and $\omega$ such that $\Ric \omega \geq r \omega$. 
Changing variables as $u^2=e^{-f}$ and applying the log-Sobolev inequality, we have 
\begin{align*}
W(\omega, f) 
&=\int_X (S_\omega u^2 +4\abs{\nabla u}^2 -u^2 \log u^2) \omega^n \\
&\geq \int_X S_\omega u^2\omega^n + (4\pi r -1 )\int_X (u^2 \log u^2) \omega ^n.  
\end{align*}
Since $S_\omega \ge n r$ and $r \geq 1/4\pi$ it yields 
$\mu(\omega) \geq nVr$ and hence $\sup_\omega \mu(\omega) \geq nV R(X)$.  
The last claim follows from Theorem \ref{mu-entropy}. 
\end{proof}  

As a consequence of \cite{KMM92} (see also \cite{Bir16}), for $n$-dimensional Fano manifolds $R(X)$ are uniformly bounded from below by a positive constant. 
The author does not know an example of Fano manifolds with $R(X) \leq 1/4\pi$. 

\subsection{Weak geodesic ray asymptotic to the flow}
Let us explain how ideas in the previous sections reprove Theorem \ref{lower bound of the H-functional}.  
It is natural to take the normalized K\"ahler-Ricci flow
\begin{equation*}
\frac{\partial}{\partial t} \omega = -\Ric \omega +\omega
\end{equation*}
in place of the inverse Monge-Amp\`ere flow. 
We again distinguish the flow $\phi_t$ from the geodesic $\phi^t$ by using the subscript. 
In terms of the normalized Ricci potential this can be described as 
\begin{equation}\label{KRF}
\frac{\partial}{\partial t} \phi = -\rho. 
\end{equation}
In fact the equation (\ref{KRF}) incorporates the slope into the H-functional in the form 
\begin{equation*}
 H(\omega_{\phi_t}) =  -\frac{d}{dt}L(\phi_t) +F(\dot{\phi}_t). 
\end{equation*}
As it was shown in \cite{P08}, \cite{PSSW09}, $H(\omega_{\phi_t})$ is non-increasing. 
Notice that in \cite{DH17} another normalization of the K\"ahler potential 
\begin{equation*}
r_t:= \phi_t -E(\phi_t)
\end{equation*}
is adopted. Our choice of $\phi_t$ is precisely equal to $\tilde{r}_t$ in their notation. 
By Perelman's uniform estimate for the Ricci potential we have $\sup_X \phi_t \leq ct +A$. 
See \cite{Pere02}, \cite{ST08} for the expoundation.   
For the Monge-Amp\`ere energy, $E(\phi_t)$ is non-decreasing from Jensen's inequality. 
In particular, the finite slope $\lim_{t \to \infty} t^{-1}E(\phi_t)$ exists. 
For the D-energy we obtain 
\begin{align*}
\frac{d}{dt}D(\phi_t)
=\frac{-1}{V}\int_X \rho(e^\rho-1)\omega_{\phi_t}^n 
=-H(\omega_{\phi_t}) +\frac{1}{V}\int_X \rho \omega_{\phi_t}^n. 
\end{align*}
Jensen's inequality shows 
\begin{equation*}
\frac{1}{V}\int_X \rho \omega_{\phi_t}^n
\leq \log \frac{1}{V}\int_X e^\rho \omega_{\phi_t}^n =0 
\end{equation*} 
so that $D(\phi_t)$ is non-increasing. 
Consequently, we may repeat the argument in subsection 3.3 to deduce the following. 

\begin{thm}[Renormalization of \cite{DH17}, Theorem 2]
Let $\phi_t$ be the normalized K\"ahler-Ricci flow and $\phi_j^t$ $(t \in [0, t_j])$ be the weak geodesic ray of the envelope form (\ref{PB envelope}) so as to connect $\phi_0$ to $\phi_{t_j}$. 
Then there exists a ray $\phi^t$ such that $\lim_{j \to \infty} d_p(\phi_j^t, \phi^t) = 0$ for each $t$. 
As a result $\phi^t$ is a constant-speed geodesic for all $(\cE^p(X, \omega_0), d_p)$. 
It satisfies $\sup_X \phi^t \leq ct+A$ and $E(\phi^t)$ constant. 
\end{thm}

\begin{proof}
The statement was originally proved for $r_t = \phi_t -E(\phi_t)$. 
Let $r_j^t$, $r^t$ be the corresponding weak geodesic. 
In the same way as subsection 3.3 we obtain the limit $\phi^t$ of $\phi_j^t$. 
It then easy to check that $\phi_j^t= r_j^t + \e_j t$ holds for 
$\e_j := t_j^{-1}(E(\phi_{t_j})-E(\phi_0))$ which converges to $\e:= \lim_{t \to \infty} t^{-1}E(\phi_t)$. We conclude $\phi^t = r^t +\e t$. 
\end{proof} 

We notice that the property $\lim_{j \to \infty} d_p(\phi_j^t, \phi^t) = 0$ follows from the Harnack estimate for the K\"ahler-Ricci flow. 
Let us extend the definition of $F(\dot{\phi}^t)$ to the above (possibly not differential) weak geodesic ray. 
First we recall: 
\begin{thm}[\cite{Dar17a}, Theorem 1] 
For the $\phi^t$ constructed from the envelope form (\ref{PB envelope}) 
we have constants $m, M$ such that   
for any $a, b \in [0, \infty)$  
\begin{itemize} 
\item[$(1)$] 
$\inf_X \frac{\phi_a-\phi_b}{a-b} = m$, 
\item[$(2)$]
$\sup_X \frac{\phi^a-\phi^b}{a-b} =M$. 
\end{itemize}
\end{thm}
Solution of the Hausdorff moment problem guarantees the following definition. 
\begin{dfn}
Let $\phi^t$ be the above weak geodesic ray constructed from the envelope form (\ref{PB envelope}), which in particular has a constant speed for any $d_p$. 
Define the Duistermatt-Heckman measure $\DH(\phi^t)$ as the unique measure supported on $[m, M]$ such that for any $p\geq 1$ 
\begin{equation*}
\bigg[ \int_\R \abs{\lambda}^p \DH(\phi^t)\bigg]^{\frac{1}{p}}= \frac{d_p(\phi^0, \phi^t)}{t} 
\end{equation*}
holds. By definition $\DH(\phi^t)$ does not depend on $t$. 
We set 
\begin{equation*}
F(\dot{\phi}^t):= -\log \int_\R e^{-\lambda} \DH(\phi^t).  
\end{equation*}
\end{dfn}
When $\phi^t$ is the weak geodesic ray associated to a test configuration, we have 
\begin{equation*}
\int_\R \abs{\lambda}^p \DH(\cX, \cL)= 
\frac{1}{V}\int_X \abs{\dot{\phi}^t}^p \omega_{\phi^t}^n 
=\frac{d_p(\phi^0, \phi^t)^p}{t^p}. 
\end{equation*}
It implies $\DH(\phi^t) = \DH(\cX, \cL)$. 
For the flow $\phi_t$ we set $\DH(\phi_t):= (\dot{\phi}_t)_*(V^{-1}\omega^n)$.  

\begin{lem}
For the above weak geodesic ray asymptotic to the normalized K\"ahler Ricci flow we have 
\begin{equation*}
F(\dot{\phi}^t) =F(\dot{\phi}_t)=0. 
\end{equation*}
\end{lem}
\begin{proof}
First we observe 
\begin{equation*}
\int_\R e^{-\lambda} \DH(\phi_t)=\frac{1}{V}\int_X e^{-\dot{\phi}_t} \omega^n  =\frac{1}{V}\int_X e^{\rho} \omega^n =1. 
\end{equation*}
Since $\DH(\phi_j^t)$ is constant in $t$, we have 
\begin{align*}
\int_\R \abs{\lambda}^p \DH(\phi_j^{t})
&=\frac{1}{t_j}\int_0^{t_j}\int_\R \abs{\lambda}^p \DH(\phi_j^{t})\\
&=\frac{d_p(\phi_0, \phi_j^{t_j})^p}{t_j}
=\frac{d_p(\phi_0, \phi_{t_j})^p}{t_j}\\
&=  \frac{1}{t_j}\int_0^{t_j} \int_\R \abs{\lambda}^p \DH(\phi_t) 
\end{align*}
for any $p \geq 1$. 
It means the identity of the probability measures:  
\begin{equation*}
\DH(\phi_j^{t})
=\frac{1}{t_j}\int_0^{t_j} \DH(\phi_t). 
\end{equation*}
In the same way we obtain $\DH(\phi^{t})= \lim_{t \to \infty}  \DH(\phi_j^t)$ from $\lim_{j \to \infty} d_p(\phi_j^t, \phi^t) = 0$. 
In particular 
\begin{equation*}
 \int_\R e^{-\lambda} \DH(\phi^t) 
= \lim_{j \to \infty}  \int_\R e^{-\lambda} \bigg[ \frac{1}{t_j} \int_0^{t_j} \DH(\phi_t) \bigg] 
= 1. 
\end{equation*}
\end{proof} 

\subsection{Multiplier ideal sheaves for the asymptotic weak geodesic ray}
Totally in parallel with section 3.3 we may further construct approximative test configurations. 
Set the $\S^1$-invariant $p_2^*\omega_0$-psh function $\Phi(x, e^{-t}) := \phi^t(x)$.  
The linear bound $\sup_X \phi \leq ct +A$ implies that $\hat{\Phi}:=\Phi+c\log\abs{\tau}$ is uniquely extended to a $p_1^*\omega_0$-psh function on $X\times \B$. 
We obtain the $\S^1$-invariant multiplier ideal sheaf $\cJ(m\hat{\Phi})$ and the normalized blow-up $\rho_m\colon \cX_m \to X\times \A^1$ with exceptional divisor $E_m$. 
The line bundle is given by 
\begin{equation}\label{line bundle of the test configuration}
\cL_m:=\rho_m^*p_1^*(-K_X) -\frac{1}{m+m_0}E_m+\frac{cm}{m+m_0}\rho_m^*\cX_{m, 0}.  
\end{equation}

\begin{thm}
Let $\phi^t$ be the above weak geodesic ray for the normalized K\"ahler-Ricci flow and $(\cX_m, \cL_m)$ be the canonical sequence of test configurations approximates $\phi^t$. 
Then we have 
\begin{equation*}
 \liminf_{m \to \infty} H(\cX_m, \cL_m)
 \geq 
 \lim_{t \to \infty} \bigg[ \frac{-L(\phi^t)}{t}+ F(\dot{\phi}^t) \bigg]. 
\end{equation*}
\end{thm} 
\begin{proof}
For the part concerned with $L^\NA(\cX_m, \cL_m)$ it is due to \cite{BBJ15}. 
The $F(\cX_m, \cL_m)$ part follows from essentially the same argument as that  for $E^\NA$.   
Indeed, for the weak geodesic ray $\Phi_m$ associated with $(\cX_m, \cL_m)$    we obtain $\Phi_m \geq \Phi-C_{m, r}$ by using local Demailly approximation. 
We again use Theorem \ref{Lidskii type inequality} to compare the $p$-moments as 
\begin{align*}
\int_\R \abs{\lambda}^p \DH(\phi_m^{t})
&=\frac{1}{t}\int_0^t\int_\R \abs{\lambda}^p \DH(\phi_m^{s})\\
&=\frac{d_p(\phi_0, \phi_m^{t})^p}{t}
\geq \frac{d_p(\phi_0, \phi^t-C_{m, r})^p}{t}. 
\end{align*}
As $t \to \infty$, just by the definition of $\DH(\phi^{t})$, the last term converges to $\int_\R \abs{\lambda}^p \DH(\phi^{t})$. 
It implies $\DH(\phi_m^t) \geq \DH(\phi^t)$ so that $F(\cX_m, \cL_m) \geq F(\phi^t)$. 
\end{proof} 
Combining all together, we obtain 
\begin{align*}
\liminf_{m \to \infty} H(\cX_m, \cL_m) 
& \geq \lim_{t \to \infty} \bigg[ \frac{-L(\phi^t)}{t}+ F(\dot{\phi}^t) \bigg] \\
& = \lim_{t \to \infty} \bigg[ \frac{-L(\phi_t)}{t}+ F(\dot{\phi}_t) \bigg]. 
\end{align*}
Finally the K\"ahler-Ricci flow equation translates the last term into the limit of 
$H(\omega_{\phi_t})$. 
It completes the proof of Theorem B. 



\end{document}